\documentclass[reqno,12pt]{amsart}
\usepackage[letterpaper, margin=1in]{geometry}
\RequirePackage{amsmath,amssymb,amsthm,graphicx,mathrsfs,url,slashed,subcaption}
\RequirePackage[usenames,dvipsnames]{xcolor}
\RequirePackage[colorlinks=true,linkcolor=Red,citecolor=Green]{hyperref}
\RequirePackage{amsxtra}
\usepackage{cancel}
\usepackage{tikz-cd}
\usepackage{xcolor}
\usepackage{soul}
\usepackage{setspace}

\title{Extensions of categoricity relative to a degree}
\author{Java Darleen Villano}
\date{\today}
\thanks{The author would like to thank their adviser Reed Solomon for providing many helpful comments in the preparation of this paper. The author was partially supported by a Focused Research Group grant from the National Science Foundation of the United States, DMS-1854355.}
\address{Department of Mathematics\\ University of Connecticut\\ Storrs, Connecticut 06269}
\email{javavill@uconn.edu}

\newcommand{\NN}{\mathbb{N}}

\newcommand{\A}{\mathcal{A}}
\newcommand{\G}{\mathcal{G}}
\newcommand{\B}{\mathcal{B}}

\renewcommand{\epsilon}{\varepsilon}
\newcommand{\<}{\langle}
\renewcommand{\>}{\rangle}

\newcommand{\lessT}{<_{\text{T}}}

\renewcommand{\phi}{\varphi}

\newcommand{\age}{\mathrm{age}}

\newtheorem{theorem}{Theorem}[section]

\newtheorem{lemma}[theorem]{Lemma}

\newtheorem{question}[theorem]{Question}

\theoremstyle{definition}
\newtheorem{definition}[theorem]{Definition}



\numberwithin{equation}{section}

\def\XXint#1#2#3{{\setbox0=\hbox{$#1{#2#3}{\int}$ }
\vcenter{\hbox{$#2#3$ }}\kern-.6\wd0}}

\usepackage[backend=bibtex,style=alphabetic,giveninits=true]{biblatex}

\renewbibmacro{in:}{}
\DeclareFieldFormat{pages}{#1}

\begin{document}
\begin{abstract}
In this paper, we apply the machinery developed in \cite{villano2024computable} to study the behavior of computable categoricity relativized to non-c.e.\ degrees. In particular, we show that we can build a computable structure which is not computably categorical but is computably categorical relative to a $1$-generic degree. Additionally, we show that other classes of structures besides directed graphs admit a computable example which can change its computable categorical behavior relative to different degrees.
\end{abstract}
\maketitle



\section{Introduction}
In computable structure theory, we are interested in effectivizing model theoretic notions and constructions. For a general background on computable structure theory, see Ash and Knight \cite{Ash2000-ASHCSA} or Montalb{\' a}n \cite{Montalban2021}. In particular, many people have examined the complexity of isomorphisms between structures within the same isomorphism type. We restrict ourselves to countable structures in a computable language and assume their domain is $\omega$. 

A computable structure $\mathcal{A}$ is \textit{computably categorical} if for any computable copy $\mathcal{B}$ of $\mathcal{A}$, there exists a computable isomorphism between $\mathcal{A}$ and $\mathcal{B}$. In this paper, we continue the investigation started in \cite{MR4291596} and \cite{villano2024computable} into the following relativized version of computable categoricity. For $X\in2^\NN$, we say that a computable structure $\mathcal{A}$ is \textit{computably categorical relative to} $X$ if for all $X$-computable copies $\mathcal{B}$ of $\mathcal{A}$, there exists an $X$-computable isomorphism between $\mathcal{A}$ and $\mathcal{B}$. 

By relativized versions of results due to Ash, Knight, Manasse, and Slaman \cite{ASH1989195}, Chisholm \cite{Chisholm90}, and Goncharov \cite{MR622606}, we have for every computable structure $\mathcal{A}$ that this relativization stabilizes on a cone above $\mathbf{0}''$. That is, either $\mathcal{A}$ is computably categorical relative to all degrees above $\mathbf{0}''$, or to no degree above $\mathbf{0}''$. Downey, Harrison-Trainor, Melnikov in \cite{MR4291596} and the author in \cite{villano2024computable} later showed that in the c.e.\ degrees, this relativization behaves chaotically in the following sense.

\begin{theorem}[Downey, Harrison-Trainor, Melnikov \cite{MR4291596}]\label{DHM result}
    There is a computable structure $\A$ and c.e. degrees $0=Y_0\lessT X_0\lessT Y_1\lessT X_1\lessT\dots$ such that
    \begin{itemize}
        \item[(1)] $\A$ is computably categorical relative to $Y_i$ for each $i$,
        \item[(2)] $\A$ is not computably categorical relative to $X_i$ for each $i$,
        \item[(3)] $\A$ is relatively computably categorical to $\mathbf{0'}$.
    \end{itemize}
\end{theorem}

\begin{theorem}[\cite{villano2024computable}]
    Let $P=(P,\leq)$ be a computable partially ordered set and let $P=P_0\sqcup P_1$ be a computable partition. Then, there exists a computable computably categorical directed graph $\G$ and an embedding $h$ of $P$ into the c.e.\ degrees where $\G$ is computably categorical relative to each degree in $h(P_0)$ and is not computably categorical relative to each degree in $h(P_1)$.
\end{theorem}

We now discuss some restrictions of the techniques in \cite{villano2024computable}. We first introduce the following definitions.

\begin{definition}
    A computable structure $\A$ is $\mathbf{d}$\textit{-computably categorical} if for all computable $\B\cong\A$, there exists a $\mathbf{d}$-computable isomorphism between $\A$ and $\B$.
\end{definition}

\begin{definition}
    A structure $\A$ has \textit{degree of categoricity} $\mathbf{d}$ if $\A$ is $\mathbf{d}$-computably categorical and for all $\mathbf{c}$, if $\A$ is $\mathbf{c}$-computably categorical, then $\mathbf{d}\leq\mathbf{c}$. A degree $\mathbf{d}$ is a \textit{degree of categoricity} if there is some structure with degree of categoricity $\mathbf{d}$.
\end{definition}

Finding a characterization of degrees of categoricity in the Turing degrees has recently been an active topic in computable structure theory. For a survey paper of development up until 2017, see Franklin \cite{franklin17}. Degrees which are \textit{not} degrees of categoricity exist, with Anderson and Csima producing several examples in \cite{AC2016}. One important example is the following.

\begin{theorem}[Anderson, Csima \cite{AC2016}]
    There is a $\Sigma_2^0$ degree that is not a degree of categoricity.
\end{theorem}

In fact, the $\Sigma_2^0$ degree that they built to witness this result is low for isomorphism.

\begin{definition}
    A degree $\mathbf{d}$ is \textit{low for isomorphism} if for every pair of computable structures $\mathcal{A}$ and $\mathcal{B}$, $\A$ is $\mathbf{d}$-computably isomorphic to $\B$ if and only if $\A$ is computably isomorphic to $\B$.
\end{definition}

There are currently no known characterizations for LFI degrees, but they are connected to the generic degrees. We quickly recall the definition of an $n$-generic set.

\begin{definition}
    A set $A$ is $n$\textit{-generic} if for all $\Sigma_n^0$ set of strings $S\subseteq 2^{<\omega}$, there exists an $m$ such that either $A\restriction m\in S$ or for all $\tau\supseteq A\restriction m$, $\tau\not\in S$. A degree $\mathbf{d}$ is $n$\textit{-generic} if it contains an $n$-generic set.
\end{definition}

\begin{theorem}[Franklin, Solomon \cite{franklinsolomon2014}]
    Every $2$-generic degree is low for isomorphism.
\end{theorem}

This tells us that for a $2$-generic degree $\mathbf{d}$, there exists \textit{no} computable structure $\A$ where $\A$ is not computably categorical but is computably categorical relative to $\mathbf{d}$, one of the first examples of a degree where techniques from \cite{villano2024computable} do not apply. Here, we will show that this is optimal in the generic degrees, that is, we can build a $1$-generic degree $\mathbf{d}$ and a computable structure $\A$ which is not computably categorical but is computably categorical relative to $\mathbf{d}$. 

\begin{theorem}\label{thm: categoricity relative to a 1-generic}
    There exists a (properly) $1$-generic $G$ such that there is a computable directed graph $\A$ where $\A$ is not computably categorical but is computably categorical relative to $G$.
\end{theorem}

The proof is a priority construction on a tree of strategies, using largely the same machinery developed in \cite{villano2024computable}. In Section \ref{section: informal strategies for 1-generic}, we
introduce informal descriptions for the strategies we need to satisfy our requirements for the construction and discuss important interactions between certain strategies. In Section \ref{section: main proof for generic result}, we detail the formal strategies, state and prove auxiliary lemmas about our construction, and state and prove the main verification lemma. We then end this paper with Section \ref{section: other classes of structures} where we discuss for which classes of structures admit or do not admit a computable example that can change its categorical behavior relative to a degree above $\mathbf{0}$.

\section{Informal strategies for Theorem \ref{thm: categoricity relative to a 1-generic}}\label{section: informal strategies for 1-generic}
We first establish the informal strategies to meet three goals for our construction: to build a $1$-generic set $G$, to make our graph $\A$ not computably categorical, and to make our graph $\A$ computably categorical relative to $G$. We will then describe the interactions that occur when we use all three strategies together, and how to resolve any issues that arise.

\subsection{Building a $1$-generic $G$}\label{section: informal building a generic}
We will build a $1$-generic $G$ via infinitely many strategies. Recall that for a set $G$ to be $1$-generic, we must have that for all $e\in\omega$, there exists an initial segment $\sigma$ of $G$ where either $\sigma\in W_e$ or for all extensions $\tau\supseteq\sigma$, $\tau\not\in W_e$. That is, either $G$ meets or avoids each c.e.\ set. 

For each $j\in\omega$, we meet the requirement
\[
R_j : (\exists\sigma\subseteq G)(\sigma\in W_j \lor (\forall\tau\supseteq\sigma)(\tau\not\in W_j)).
\]

We will define sets $G[s]$, where $s$ is a stage number, which are a $\Delta_2^0$ approximation to our $1$-generic $G$. Let $\alpha$ be the highest priority $R$-strategy. When it is first eligible to act at a stage $s_0$, $\alpha$ will define a parameter $n_\alpha>0$ large (and so in particular, $n_\alpha>s_0$), and its goal is to find an extension $\tau_0\supseteq G[s_0]\restriction n_\alpha$ where $\tau_0\in W_j$. At each $\alpha$-stage, it searches for such an extension.

If $\alpha$ never finds such an extension, then we have that all extensions of $G[s_0]\restriction n_\alpha$ avoid $W_j$ and $\alpha$ succeeds trivially. If $\alpha$ finds an extension $\tau_0$ such that $\tau_0\in W_j$ at a stage $s_1\geq s_0$, then $\alpha$ will define $G[s_1]=\tau_0^\frown 0^\omega$ and will initialize all lower priority strategies. In particular, all lower priority $R$-strategies will now have to redefine new, larger parameters (and so these parameters will be bigger than $|\tau_0|$) and must now use $G[s_1]$ as the current approximation of $G$ at the end of stage $s_1$ for all stages $s\geq s_1$. We also have that $G[s_1]\restriction n_\alpha=G[s_0]\restriction n_\alpha$, so $\alpha$ does not cause changes on numbers below $n_\alpha$.

If all $R_j$-strategies succeeded, then we define $G=\lim\limits_{s\to\infty} G[s]$ to be our $1$-generic set. This limit exists by the observation in the previous paragraph.

\subsection{Making $\A$ not computably categorical}\label{section: informal not c.c.}
We will build our directed graph $\A$ in stages. At stage $s=0$, we set $\A=\emptyset$. Then, at stage $s>0$, we add two new connected components to $\A[s]$ by adding the root nodes $a_{2s}$ and $a_{2s+1}$ for those components, and attaching to each node a $2$-loop (a cycle of length $2$). We then attach a $(5s+1)$-loop to $a_{2s}$ and a $(5s+2)$-loop to $a_{2s+1}$. This gives us the configuration of loops:
    \begin{align*}
        a_{2s} &: 2, 5s+1  \\
        a_{2s+1} &: 2, 5s+2.
    \end{align*}
The connected component consisting of the root node $a_{2s}$ with its attached loops will be referred to as the $2s$\textit{th connected component} of $\A$. During the construction, we might add more loops to connected components of $\A$, which causes them to have the following configuration:
    \begin{align*}
        a_{2s} &: 2, 5s+1, 5s+2, 5s+3  \\
        a_{2s+1} &: 2, 5s+1, 5s+2, 5s+4
    \end{align*}
This configuration, if certain interactions occur during the construction, can also transform into the following:
    \begin{align*}
        a_{2s} &: 2, 5s+1, 5s+2, 5s+3, 5s+4  \\
        a_{2s+1} &: 2, 5s+1, 5s+2, 5s+3, 5s+4.
    \end{align*}

Note that the last configuration has that the $2s$th and $(2s+1)$st components of $\A$ are isomorphic, which may be necessary as a result of a special interaction between strategies of all three types of requirements in this construction (see \ref{subsection: interactions}).

In order to make $\A$ not computably categorical, it is sufficient to construct a computable copy $\B$ such that for all $e\in\omega$, the computable function $\Phi_e$ is not an isomorphism between $\A$ and $\B$. 

Similarly to $\A$, we build the directed graph $\B$ in stages. At stage $s=0$, we set $\B=\emptyset$. At stage $s>0$, we add root nodes $b_{2s}$ and $b_{2s+1}$ to $\B$ and attach to each one a $2$-loop. Next, we attach a $(5s+1)$-loop to $b_{2s}$ and a $(5s+2)$-loop to $b_{2s+1}$. However, throughout the construction, we may add new loops to specific components of $\B$. For the $2s$th and $(2s+1)$st components of $\B$, we have three possible final configurations of the loops. If we never start the process of diagonalizing using these components, then they will remain the same forever:
    \begin{align*}
        b_{2s} : 2, 5s+1  \\
        b_{2s+1} : 2, 5s+2. 
    \end{align*}
If we use these components to diagonalize against a computable map $\Phi_e$, they will end in the following configuration:
    \begin{align*}
        b_{2s} : 2, 5s+1, 5s+2, 5s+4  \\
        b_{2s+1} : 2, 5s+1, 5s+2, 5s+3. 
    \end{align*}
If these components were used to diagonalize against $\Phi_e$ and then certain interactions occur between all three types of requirements, these components may have the final configuration:
    \begin{align*}
        b_{2s} : 2, 5s+1, 5s+2, 5s+3, 5s+4  \\
        b_{2s+1} : 2, 5s+1, 5s+2, 5s+3, 5s+4. 
    \end{align*}

For all $e\in\omega$, we meet the following requirement
\[
P_e : \text{$\Phi_e:\A\to\B$ is not an isomorphism}.
\]

To satisfy this requirement, we will diagonalize against $\Phi_e$. Let $\alpha$ be a $P_e$-strategy.

When $\alpha$ is first eligible to act, it picks a large number $n_\alpha$, and for the rest of this strategy, let $n=n_\alpha$. This parameter indicates which connected components of $\B$ will be used in the diagonalization. At future stages, $\alpha$ checks if $\Phi_e$ maps the $2n$th and $(2n+1)$st connected component of $\A$ to the $2n$th and $(2n+1)$st connected component of $\B$, respectively. If not, $\alpha$ does not take any action. If $\alpha$ sees such a computation, it acts in the following way.

At this point, our connected components in $\A[s]$ and $\B[s]$ are as follows:
    \begin{align*}
        a_{2n} : 2, 5n+1 & \ \ \ \ b_{2n} : 2, 5n+1 \\
        a_{2n+1} : 2, 5n+2 & \ \ \ \ b_{2n+1} : 2, 5n+2.
    \end{align*}
Since $\Phi_e$ looks like a potential isomorphism between $\A$ and $\B$, $\alpha$ will now take action to eventually force the true isomorphism to match $a_{2n}$ with $b_{2n+1}$ and to match $a_{2n+1}$ with $b_{2n}$.

$\alpha$ adds a $(5n+2)$- and $(5n+3)$-loop to $a_{2n}$ and a $(5n+1)$- and $(5n+4)$-loop to $a_{2n+1}$ in $\A[s]$. It also attaches a $(5n+2)$- and $(5n+4)$-loop to $b_{2n}$ and a  $(5n+1)$- and $(5n+3)$-loop to $b_{2n+1}$ in $\B$[s]. Our connected components in $\A[s]$ and in $\B[s]$ are now:
    \begin{align*}
        a_{2n} : 2, 5n+1, 5n+2, 5n+3 & \ \ \ \ b_{2n} : 2, 5n+1, 5n+2, 5n+4 \\
        a_{2n+1} : 2, 5n+1, 5n+2, 5n+4 & \ \ \ \ b_{2n+1} : 2, 5n+1, 5n+2, 5n+3.
    \end{align*}

For all higher priority $S$-strategies $\beta$ such that $\beta^\frown\<\infty\>\subseteq\alpha$, $\alpha$ enumerates the use $u_{\beta,n_\alpha}$ into $G$ and sets $n_\beta=n_\alpha$. This enumeration occurs because if we have that $\beta^\frown\<\infty\>\subseteq\alpha$, then $\alpha$ believes $\beta$ will define a total function $f_\beta^G$ on $\A$. Therefore, $\alpha$ doesn't restart when $\beta$ extends its definition of $f_\beta^G$ and so this map may be defined on the $2n_\alpha$th and $(2n_\alpha+1)$st components of $\A$ when $\alpha$ acts. In this case, $f_\beta^G[s-1]$ maps the $2n_\alpha$th and $(2n_\alpha+1)$st components of $\A$ to their copies in $\mathcal{M}_i^G$ with some use $u_{\beta,n_\alpha}$. We will choose this use carefully so that we know $u_{\beta,n_\alpha}\not\in G[s-1]$, and in fact, $u_{\beta,n_\alpha}\not\in G[t]$ for all $t<s$. This allows $\alpha$ to enumerate the use $u_{\beta,n_\alpha}$ into $G[s]$ to destroy the $f_\beta^G$ computation on the $2n_\alpha$th and $(2n_\alpha+1)$st components in $\A$.

We will refer to these two actions as $\alpha$ \textit{issuing a challenge} to $\beta$. $\alpha$ now takes the outcome $w_2$. If by the next $\alpha$-stage we have that $\alpha$ has not been initialized, then it takes the success outcome $s$. By $\alpha$'s actions above, we have that $\Phi_e(a_{2n})=b_{2n}$ and $\Phi_e(a_{2n+1})=b_{2n+1}$, and so $\Phi_e$ fails to be a computable isomorphism between $\A$ and $\B$.

\subsection{Being computably categorical relative to $G$}\label{section: informal categoricity strategy}

Since we want to make $\A$ computably categorical relative to our $1$-generic $G$, we must define embeddings using $G$ as an oracle. Additionally, we are looking at (partial) $G$-computable directed graphs $\mathcal{M}_i^G$ with domain $\omega$ whose edge relation is given by $\Phi_i^G$. Any changes in initial segments of $G$ can cause changes for both our partial $G$-computable embeddings and in $\mathcal{M}_i^G$ throughout the construction.

We define the following terms to keep track of certain finite subgraphs which appear and may remain in $\mathcal{M}_i^G$ throughout our construction.

\begin{definition}
    Let $C_0$ and $C_1$ be isomorphic finite distinct subgraphs of $\mathcal{M}_i^G[s]$. The \textit{age of} $C_0$ is the least stage $t\leq s$ such that all edges in $C_0$ appear in $\mathcal{M}_i^G[t]$ and remained in $\mathcal{M}_i^G[s']$ for all $t\leq s'\leq s$, denoted by $\age(C_0)$. We say that $C_0$ is \textit{older than} $C_1$ when $\age(C_0)\leq\age(C_1)$.

    We say that $C_0$ is the \textit{oldest} if for all finite distinct subgraphs $C\cong C_0$ of $\mathcal{M}_i^G[s]$, $\age(C_0)\leq\age(C)$.
\end{definition}

\begin{definition}
    Let $C_0=\<a_0,a_1,\dots,a_k\>$ and $C_1=\<b_0,b_1,\dots,b_k\>$ be isomorphic finite distinct subgraphs of $\mathcal{M}_i^G[s]$ with $a_0<a_1<\dots<a_k$ and $b_0<b_1<\dots<b_k$. We say that $C_0<_{\text{lex}} C_1$ if for the least $j$ such that $a_j\neq b_j$, $a_j<b_j$.

    We say that $C_0$ is the \textit{lexicographically least} if for all finite distinct subgraphs $C\cong C_0$ of $\mathcal{M}_i^G[s]$, $C_0<_{\text{lex}} C$.
\end{definition}

If $\A\cong\mathcal{M}_i^G$, then we need to build a $G$-computable isomorphism between these graphs. For each $i\in\omega$, we meet the following requirement
    \[
    S_i : \text{if $\A\cong\mathcal{M}_i^G$, then there exists a $G$-computable isomorphism $f_i^G:\A\to\mathcal{M}_i^G$}.
    \]

We will show in the verification that if $\A\cong\mathcal{M}_i^G$, ``true'' copies of components from $\A$ will eventually appear and remain in $\mathcal{M}_i^G$ (and thus become the oldest and lex-least finite subgraph which is isomorphic to a component in $\A$), and so our $S_i$-strategy below will be able to define the correct $G$-computable isomorphism between the two graphs. 

Let $\alpha$ be an $S_i$-strategy. When $\alpha$ is first eligible to act, it sets its parameter $n_\alpha=0$ and defines $f_\alpha^G$ to be the empty map. Once $\alpha$ has defined $n_\alpha$, then when $\alpha$ acted at the previous $\alpha$-stage $s_0$, we have the following situation: 
\begin{enumerate}
    \item For each $m<n_\alpha$, $f_\alpha^G[s_0]$ maps the $2m$th and $(2m+1)$st components of $\A[s_0]$ to isomorphic copies in $\mathcal{M}_i^G[s_0]$.
    \item For $m<n_\alpha$, let $l_m$ be the maximum $\Phi_i^G[s_0]$-use for the loops in the copies in $\mathcal{M}_i^G[s_0]$ for the $2m$th and $(2m+1)$st components in $\A$. We can assume that if $m_0<m_1<n_\alpha$, then $l_{m_0}<l_{m_1}$.
    \item For $m<n_\alpha$, let $u_{\alpha,m}$ be the $f_\alpha^G[s_0]$-use for the mapping of the $2m$th and $(2m+1)$st components of $\A$. This use will be constant for all elements in these components.
    \item By construction, we will have that $l_m<u_{\alpha,k}$ for all $m\leq k<n_\alpha$.
\end{enumerate}

Suppose $\alpha$ is acting at stage $s$ and that $n_\alpha>0$ and let $s_0$ be the previous $\alpha$-stage in the construction. We now break into cases.

If $\alpha$ took an outcome $w_n$ at $s_0$, then no strategy could have changed $G$ below the use $u_{\alpha,n_\alpha-1}$, and so the value of $n_\alpha$ remains unchanged. 

If instead we have that $\alpha$ took the $\infty$ outcome at $s_0$, then $G$ may have changed underneath $u_{\alpha,n_\alpha-1}$. We will show in the verification that the only strategies that can change $G$ below this use are $P$ or $R$-strategies $\beta$ such that $\beta\supseteq\alpha^\frown\<\infty\>$. If $G[s]\restriction(u_{\alpha,n_\alpha-1}+1)\neq G[s_0]\restriction(u_{\alpha,n_\alpha-1}+1)$, then some $\beta\supseteq\alpha^\frown\<\infty\>$ caused a change at stage $s_0$ after $\alpha$ acted. Furthermore, at most one such $\beta$ can cause this change at stage $s_0$, as every other strategy extending $\beta^\frown\<s\>$ will define new and large witnesses for the remainder of stage $s_0$. 

If $\beta$ is a $P$-strategy which changed $G$ under $u_{\alpha,n_\alpha-1}$, then it added diagonalizing loops to the $2n_\beta$th and $(2n_\beta+1)$st components in $\A$. Assuming that $n_\beta<n_\alpha$, then $\alpha$ has already defined $f_\alpha^G[s_0]$ on these components. Hence, $\beta$ enumerates the use $u_{\alpha,n_\beta}$ into $G$ to destroy this computation. To account for this change, $\alpha$ redefines $n_\alpha$ to be the least $m$ such that $f_\alpha^G$ is not currently defined on the $2m$th and $(2m+1)$st components of $\G$.

If instead we have that the $\beta$ is an $R$-strategy, then suppose $\beta$ redefines $G$ by setting $G[s_0]=\tau^\frown 0^\omega$. This definition may have changed $G$ on numbers as small as $n_\alpha$ and hence may have injured previously defined $f_\alpha^G$ computations. So, at stage $s$, $\alpha$ resets $n_\alpha$ to be the least $m$ such that $f_\alpha^G$ is not currently defined on the $2m$th and $(2m+1)$st components of $\G$.



We now carry out the main action of the $S_i$-strategy: we check whether we can extend $f_\alpha^G[s-1]$ to the $2n_\alpha$th and $(2n_\alpha+1)$st components of $\A[s]$. Search for isomorphic copies in $\mathcal{M}_i^G[s]$ of these components. If there are multiple copies in $\mathcal{M}_i^G[s]$, choose the oldest such copy to map to, and if there are multiple equally old copies, choose the lexicographically least oldest copy. If there are no copies in $\mathcal{M}_i^G[s]$, then keep the value of $n_\alpha$ the same, $f_\alpha^G$ unchanged, and let the next requirement act. Otherwise, extend $f_\alpha^G[s-1]$ to $f_\alpha^G[s]$ to include the $2n_\alpha$th and $(2n_\alpha+1)$st components of $\A$ with a large use $u_{\alpha,n_\alpha}$. Since $u_{\alpha,n_\alpha}$ was chosen large, we have that $u_{\alpha, n_\alpha}>l_k$ for all $k\leq n_\alpha$. Increment $n_\alpha$ by $1$ and go to the next requirement. If $\alpha$ had to redefine $n_\alpha$ because it was challenged by an $S$-strategy $\beta$ of higher priority, then $\alpha$ continues to take finitary outcomes until the value of $n_\alpha$ exceeds $n_\beta$.
    
If $\A\cong\mathcal{M}_i^G$, then for each $n$, eventually the real copies of the $2n$th and $(2n+1)$st components of $\A$ will appear in $\mathcal{M}_i^G$. Moreover, they will eventually be the oldest and lexicographically least copies in $\mathcal{M}_i^G$. After $G$ has settled down on the the maximum true $G$-use on the edges in the loops in these $\mathcal{M}_i^G$ components, we will define $f_\alpha^G$ correctly on these components with a large use $u_{\alpha,n}$. Therefore, eventually our map $f_\alpha^G$ is never injured again on the $2n$th and $(2n+1)$st components. It follows that if $A\cong\mathcal{M}_i^G$, then $f_\alpha^G$ will be an embedding of $\A$ into $\mathcal{M}_i^G$ which can be extended to a $G$-computable isomorphism defined on all of $\A$. 

\subsection{An interaction caused by genericity}\label{subsection: interactions}

The following interaction arises because we are building a $1$-generic. The fact that initial segments of $G$ can change several times throughout the construction requires our isolated strategies to be more flexible than what was done previously in \cite{villano2024computable}, and we outline the needed changes to affected strategies below.

Let $\alpha$ be an $S_i$-strategy, $\beta$ be an $R_j$-strategy, and $\gamma$ be a $P_e$-strategy where $\alpha^\frown\<\infty\>\subseteq\beta^\frown\<w_1\>\subseteq\gamma$, where $\beta^\frown\<w_1\>$ indicates that $\beta$ defined its parameter $n_\beta$ and took the waiting outcome at the last $\beta$-stage. Suppose that $n_\beta$, at a stage $s_0$, and $n_\gamma$ have been defined and that $n_\beta<n_\gamma$. Suppose that at stage $s_1$, $\alpha$ is able to define $f_\alpha^G[s_1]$ with a use  $u_{\alpha,n_\gamma}$ on the $2n_\gamma$th and $(2n_\gamma+1)$st components of $\A$, with the configuration of the loops in $\A[s_1]$ and $\mathcal{M}_i^G[s_1]$ being:
    \begin{align*}
        a_{2n_\gamma} : 2, 5n_\gamma+1 & \ \ \ \ c : 2, 5n_\gamma+1 \\
        a_{2n_\gamma+1} : 2, 5n_\gamma+2 & \ \ \ \ d : 2, 5n_\gamma+2.
    \end{align*}
Note that $f_\alpha^G[s_1]$ on these two $\A$-components is being protected by the initial segment $G[s_1]\restriction(u_{\alpha,n_\gamma}+1)$. Then, at the end of stage $s_1$, $\alpha$ takes the $\infty$ outcome, and $\beta$ is eligible to act. Suppose that $\beta$ continues to take the $w_1$ outcome, and so $\gamma$ is eligible to act at a stage $s_2>s_1$. Suppose that $\gamma$ sees that the map $\Phi_e[s_2]$ maps its chosen $\A$-components isomorphically to their copies in $\B[s_2]$, and thus begins to diagonalize by adding new loops to the following $\A$-components and $\B$-components:
    \begin{align*}
        a_{2n_\gamma} : 2, 5n_\gamma+1, 5n_\gamma+2, 5n_\gamma+3 & \ \ \ \ b_{2n_\gamma} : 2, 5n_\gamma+1, 5n_\gamma+2, 5n_\gamma+4 \\
        a_{2n_\gamma+1} : 2, 5n_\gamma+1, 5n_\gamma+2, 5n_\gamma+4 & \ \ \ \ b_{2n_\gamma+1} : 2, 5n_\gamma+1, 5n_\gamma+2, 5n_\gamma+3.
    \end{align*}
When $\gamma$ adds these loops, it enumerates $u_{\alpha,n_\gamma}$ into $G[s_2]$, and so $G[s_2]\restriction(u_{\alpha,n_\gamma}+1)\neq G[s_1]\restriction(u_{\alpha,n_\gamma}+1)$ and $f_\alpha^G[s_1]$ disappears on the the $2n_\gamma$th and $(2n_\gamma+1)$st components of $\A$. Let $s_3\geq s_2$ be a stage such that newly added loops first appeared in $\mathcal{M}_i^G[s_3]$ in the following positions:
    \begin{align*}
        a_{2n_\gamma} : 2, 5n_\gamma+1, 5n_\gamma+2, 5n_\gamma+3 & \ \ \ \ c : 2, 5n_\gamma+1, 5n_\gamma+2, 5n_\gamma+4 \\
        a_{2n_\gamma+1} : 2, 5n_\gamma+1, 5n_\gamma+2, 5n_\gamma+4 & \ \ \ \ d : 2, 5n_\gamma+1, 5n_\gamma+2, 5n_\gamma+3.
    \end{align*}
$\alpha$ is now able to recover its map $f_\alpha^G$ on these $\A$-components with a new large use $u_{\alpha,n_\gamma}'$ by mapping $a_{2n_\gamma}$ to $d$, $a_{2n_\gamma+1}$ to $c$, and all nodes in each component to their respective copies in $\mathcal{M}_i^G[s_3]$. Note that this new map $f_\alpha^G[s_3]$ on these $\A$-components is being protected by $G[s_3]\restriction(u_{\alpha,n_\gamma}'+1)$. At the end of stage $s_3$, $\alpha$ takes the $\infty$ outcome and suppose $\beta$ is now eligible to act again. Now, suppose that $\beta$ has found an extension $\tau_{n_\beta}$ of $G[s_0]\restriction n_\beta$ which is in $W_j[s_4]$ where $s_4>s_3$. $\beta$ then extends $G[s_0]\restriction n_\beta$ to $\tau_{n_\beta}$ and defines $G[s_4]=\tau_{n_\beta}^\frown0^\omega$. Furthermore, suppose that $G[s_4]\restriction(u_{\alpha,n_\gamma}+1)=G[s_1]\restriction(u_{\alpha,n_\gamma}+1)$, and so the original map $f_\alpha^G[s_1]$ is restored on the $2n_\gamma$th and $(2n_\gamma+1)$st components of $\A$. $\beta$ taking the success outcome initializes $\gamma$ since $\gamma\supseteq\beta^\frown\<w_1\>$, however because $\A$ and $\B$ need to be computable directed graphs, the loops added at stage $s_2$ must remain in both graphs. In particular, if the true outcome for $\mathcal{M}_i^G$ is to have the following connected components
    \begin{align*}
        c : 2, 5n_\gamma+1, 5n_\gamma+2, 5n_\gamma+4 \\
        d : 2, 5n_\gamma+1, 5n_\gamma+2, 5n_\gamma+3,
    \end{align*}
then because $f_\alpha^G[s_1]$ was restored on the two $\A$-components, it is now incorrect on the $G$-computable embedding from $\A$ into $\mathcal{M}_i^G$ since $f_\alpha^G[s_1](a_{2n_\gamma})=c$ and $f_\alpha^G[s_1](a_{2n_\gamma+1})=d$, in particular:
    \begin{align*}
        a_{2n_\gamma} : 2, 5n_\gamma+1, 5n_\gamma+2, 5n_\gamma+3 & \ \ \ \ c : 2, 5n_\gamma+1, 5n_\gamma+2, 5n_\gamma+4 \\
        a_{2n_\gamma+1} : 2, 5n_\gamma+1, 5n_\gamma+2, 5n_\gamma+4 & \ \ \ \ d : 2, 5n_\gamma+1, 5n_\gamma+2, 5n_\gamma+3.
    \end{align*}

We have then that if $\alpha^\frown\<\infty\>$ is on the true path, it will define an incorrect $G$-computable embedding from $\A$ into $\mathcal{M}_i^G$. To resolve this issue, after $\beta$ finds its extension, we will add an additional step to the $R_j$-strategy where we homogenize any affected $\A$-components defined so far, namely the $2n_\gamma$th and $(2n_\gamma+1)st$ components of $\A$ in this example, in the following way:
    \begin{align*}
        a_{2n_\gamma} : 2, 5n_\gamma+1, 5n_\gamma+2, 5n_\gamma+3, 5n_\gamma+4 \\
        a_{2n_\gamma+1} : 2, 5n_\gamma+1, 5n_\gamma+2, 5n_\gamma+3, 5n_\gamma+4.
    \end{align*}
Then, if the newly added loops reappear in the corresponding copies of each component in $\mathcal{M}_i^G$ and remain, it does not matter which positions they appear in since $f_\alpha^G[s_1]$ will be correct no matter what. If other cycles or the loops never appear again after $\beta$'s success, then we have that $\A\not\cong\mathcal{M}_i^G$ and so $\alpha$ trivially wins. Furthermore, the path now goes to the left of $\gamma$, initializing $\gamma$, and so we do not hurt our construction by undoing $\gamma$'s diagonalization by homogenizing.

Additionally, since we need that $\A\cong\B$, we must also homogenize the corresponding components in $\B$ as well:
    \begin{align*}
        b_{2n_\gamma} : 2, 5n_\gamma+1, 5n_\gamma+2, 5n_\gamma+3, 5n_\gamma+4 \\
        b_{2n_\gamma+1} : 2, 5n_\gamma+1, 5n_\gamma+2, 5n_\gamma+3, 5n_\gamma+4.
    \end{align*}

Lastly, we will also homogenize the $\A$-components which received diagonalizing loops from $P$-strategies which are to the right of the current true path.

\section{Proof of Theorem \ref{thm: categoricity relative to a 1-generic}}\label{section: main proof for generic result}

\subsection{Requirements}
We have three requirements for our construction:
\[
R_j : (\exists\sigma\subseteq G)(\sigma\in W_j \lor (\forall\tau\supseteq\sigma)(\tau\not\in W_j));
\]
\[
P_e : \text{$\Phi_e:\A\to\B$ is not an isomorphism};
\]
\[
S_i : \text{if $\A\cong\mathcal{M}_i^G$, then there exists a $G$-computable isomorphism $f_i^G:\A\to\mathcal{M}_i^G$}.
\]


\subsection{Construction}
Let $\Lambda=\{\infty<_\Lambda\dots<_\Lambda s <_\Lambda w_2 <_\Lambda w_1<_\Lambda w_0\}$ be the set of outcomes, and let $T=\Lambda^{<\omega}$ be our tree of strategies. The construction will be performed in $\omega$ many stages $s$.

We define the \textit{current true path} $\pi_s$, the longest strategy eligible to act at stage $s$, inductively. For every $s$, $\lambda$, the empty string, is eligible to act at stage $s$. Suppose the strategy $\alpha$ is eligible to act at stage $s$. If $|\alpha|<s$, then follow the action of $\alpha$ to choose a successor $\alpha^\frown\<o\>$ on the current true path. If $|\alpha|=s$, then set $\pi_s=\alpha$. For all strategies $\beta$ such that $\pi_s <_L\beta$, initialize $\beta$ (that is, set all parameters associated to $\beta$ to be undefined). Also, if $\pi_s<_L\beta$ and $\beta$ is a $P$-strategy such that the $2n_\beta$th and $(2n_\beta+1)$st components of $\A$ and $\B$ have been diagonalized (that is, the $(5n+3)$- and $(5n+4)$-loops have been added), then homogenize these components in $\A$ and $\B$. If $\beta <_L \pi_s$ and $|\beta|<s$, then $\beta$ retains the same values for its parameters.

We will now give formal descriptions of each strategy and their outcomes in the construction.

\subsection{$R_j$-strategies and outcomes}
We first cover the $R_j$-strategies used to build our $1$-generic set $G$. Let $\alpha$ be an $R_j$-strategy eligible to act at stage $s$.

\textit{Case 1}: If $\alpha$ is acting for the first time at stage $s$ or has been initialized since the last $\alpha$-stage, it defines $m_\alpha=\max\{n_\beta : n_\beta \ \text{defined for $\beta\subset\alpha$}\}$ only when $n_\beta$ has been defined for all $\beta\subset\alpha$. Once $m_\alpha$ has been defined, $\alpha$ waits to see if $f_\gamma^G$ converges on the $2m_\alpha$th and $(2m_\alpha+1)$st components of $\A$ for all $S$-strategies $\gamma^\frown\<\infty\>\subseteq\alpha$. Until $f_\gamma^G$ converges on all of those components, it remains in \textit{Case 1} by taking the $w_0$ outcome and does not define its parameter $n_\alpha$.

\textit{Case 2}: If $\alpha$ took the $w_0$ outcome at the previous $\alpha$-stage and all maps $f_\gamma^G$ have converged on the $2m_\alpha$th and $(2m_\alpha+1)$st components of $\A$ for $S$-strategies $\gamma^\frown\<\infty\>\subseteq\alpha$, it defines its parameter $n_\alpha$ to be large. In particular, $n_\alpha$ is greater than all of the uses of $f_\gamma^G$ on the $2n_\beta$th and $(2n_\beta+1)$st components for all $P$- and $R$-strategies $\beta\subset\alpha$ (since $\alpha$ won't define $n_\alpha$ until all the $n_\beta$'s are defined). Then, take outcome $w_1$.

\textit{Case 3}: If $\alpha$ took the $w_1$ outcome at the end of the previous $\alpha$-stage, check if there is an extension $\tau$ of $G[s-1]\restriction n_\alpha$ such that $\tau\in W_j[s]$. If not, take the $w_1$ outcome. If such a $\tau$ is found, define $G[s]=\tau^\frown0^\omega$ and take outcome $w_2$.

\textit{Case 4}: If $\alpha$ took the $w_2$ outcome the last time it was eligible to act and has not been initialized, take outcome $s$.

\textit{Case 5}: If $\alpha$ took the $s$ outcome the last time it was eligible to act and has not been initialized, continue taking outcome $s$.



\subsection{$P_e$-strategies and outcomes}
Let $\alpha$ be a $P_e$-strategy eligible to act at stage $s$.

\textit{Case 1}: If $\alpha$ is first eligible to act at stage $s$ or has been initialized, it defines $m_\alpha=\max\{n_\beta : n_\beta \ \text{defined for $\beta\subset\alpha$}\}$ only when $n_\beta$ has been defined for all $\beta\subset\alpha$. Once $m_\alpha$ has been defined, $\alpha$ waits to see if $f_\gamma^G$ converges on the $2m_\alpha$th and $(2m_\alpha+1)$st components of $\A$ for all $S$-strategies $\gamma^\frown\<\infty\>\subseteq\alpha$. Until $f_\gamma^G$ converges on all of those components, $\alpha$ remains in \textit{Case 1} by taking the $w_0$ outcome and does not define define its parameter $n_\alpha$.

\textit{Case 2}: If $\alpha$ took the $w_0$ outcome at the previous $\alpha$-stage and all maps $f_\gamma^G$ have converged on the $2m_\alpha$th and $(2m_\alpha+1)$st components of $\A$ for $S$-strategies $\gamma^\frown\<\infty\>\subseteq\alpha$, it defines the parameter $n_\alpha=n$ to be large and takes outcome $w_1$.

\textit{Case 3}: If $\alpha$ took the $w_1$ outcome at the last $\alpha$-stage, check whether $\Phi_e[s]$ maps the $2n$th and $(2n+1)$st components of $\A$ isomorphically into $\B$. If not, take outcome $w_1$.

If so, add a $(5n+2)$- and $(5n+3)$-loop to $a_{2n}$ and a $(5n+1)$- and $(5n+4)$-loop to $a_{2n+1}$ in $\A[s]$. Add a $(5n+2)$- and $(5n+4)$-loop to $b_{2n}$ and a $(5n+1)$- and $(5n+3)$-loop to $b_{2n+1}$ in $\B[s]$. If a map $f_\beta^G$ had already been defined on these components with a use $u_{\beta,n}$ by an $S$-strategy $\beta$ where $\beta^\frown\<\infty\>\subseteq\alpha$, enumerate $u_{\beta,n}$ into $G$, and issue a challenge to each such $S$-strategy. Take outcome $w_2$.

\textit{Case 4}: If $\alpha$ took the $w_2$ outcome at the last $\alpha$-stage and has not been initialized, take outcome $s$.

\textit{Case 5}: If $\alpha$ took the $s$ outcome at the previous $\alpha$-stage and has not been initialized, take outcome $s$ again.

\subsection{$S_i$-strategies and outcomes}
Let $\alpha$ be an $S_i$-strategy eligible to act at stage $s$.

\textit{Case 1}: If $\alpha$ is acting for the first time or has been initialized since the last $\alpha$-stage, set $n_\alpha=0$, define $f_\alpha^G[s]$ to be the empty function, and take the $w_0$ outcome.

\textit{Case 2}: If we are not in \textit{Case 1}, let $t$ be the last $\alpha$-stage and we break into the following subcases.

\textit{Subcase 1}: $\alpha$ is not currently challenged by a $P$-strategy and one of the following conditions holds:
\begin{enumerate}
    \item $\alpha$ took the $w_{n_\alpha}$ outcome at stage $t$,
    \item $\alpha$ took the $\infty$ outcome at stage $t$ but no $P$ or $R$-strategy $\beta\supseteq\alpha^\frown\<\infty\>$ took the $w_1$ outcome at $t$, or
    \item $\alpha$ took the $\infty$ outcome at $t$ and an $R$-strategy $\beta\supseteq\alpha^\frown\<\infty\>$ took the $w_1$ outcome at $t$ but $G[s-1]\restriction(u_{\alpha,n_\alpha-1}+1)=G[t]\restriction(u_{\alpha,n_\alpha-1}+1)$.
\end{enumerate}
In this case, $\alpha$ passes to the module for checking for extensions of $f_\alpha^G$ below.

\textit{Subcase 2}: $\alpha$ is currently challenged by a $P$-strategy $\beta\supseteq\alpha^\frown\<\infty\>$. In the verification, we will show that $\beta$ is unique.

If $\alpha$ was challenged by $\beta$ at stage $t$ (that is, $s$ is the first $\alpha$-stage since $\alpha$ was challenged), then set $n_\alpha$ to be the least $m$ such that $f_\alpha^G[s-1]$ is not defined on the $2m$th and $(2m+1)$st components of $\A$. In the verification, we will show that $n_\alpha\leq n_\beta$. Go to the module for checking for extensions of $f_\alpha^G$. If $f_\alpha^G$ is extended to the $2n_\alpha$th and $(2n_\alpha+1)$st components, then increment $n_\alpha$. If $n_\alpha>n_\beta$, take the $\infty$ outcome and remove the challenge on $\alpha$. Otherwise, take the $w_{n_\alpha}$ outcome and the challenge on $\alpha$ remains. 

If $\alpha$ was challenged before stage $t$, then $\alpha$ acts in the same way except it skips the initial redefining of $n_\alpha$.


\textit{Subcase 3}: $\alpha$ is not currently challenged by a $P$-strategy, but $\alpha$ took the $\infty$ outcome at stage $t$, an $R$-strategy $\beta\supseteq\alpha^\frown\<\infty\>$ took outcome $w_1$ at stage $t$, and $G[s-1]\restriction(u_{\alpha,n_\alpha-1}+1)\neq G[t]\restriction(u_{\alpha,n_\alpha-1}+1)$. Reset $n_\alpha$ to be the least $m$ such that $f_\alpha^G[s-1]$ is not defined on the $2m$th and $(2m+1)$st components of $\A$. Go to the module for checking for extensions of $f_\alpha^G$.


\textit{Module for checking for extensions of $f_\alpha^G$}: $\alpha$ searches for the oldest and lex-least copies of the $2n_\alpha$th and $(2n_\alpha+1)$st components in $\mathcal{M}_i^G[s]$. If no copies are found, leave $f_\alpha^G$ and $n_\alpha$ unchanged and take outcome $w_{n_\alpha}$. If $\alpha$ is challenged, then it remains challenged. Otherwise, if such copies are found, extend $f_\alpha^G$ by mapping the $2n_\alpha$th and $(2n_\alpha+1)$st components in $\A$ onto their copies in $\mathcal{M}_i^G[s]$. Define $u_{\alpha,n_\alpha}$ large and set it as the use of the newly defined computations of $f_\alpha^G$. Increment $n_\alpha$ by $1$ and take the $\infty$ outcome (unless $\alpha$ is challenged and we still have that $n_\alpha\leq n_\beta$, in which case, take the $w_{n_\alpha}$ outcome).

\subsection{Verification}
We begin with proving several auxiliary lemmas based on key observations of the construction. Afterwards, we will prove the main verification lemma. It is easy to see that the first two lemmas below are true.

\begin{lemma}\label{lemma: embedding of a graph into itself is an isomorphism}
    If $f:\A\to\A$ is an embedding of $\A$ into itself, then $f$ is an isomorphism.
\end{lemma}

\begin{lemma}\label{lemma: graph isomorphisms}
    If $\A\cong\mathcal{M}_i^G$ for a $G$-computable directed graph $\mathcal{M}_i^G$ and $f^G:\A\to\mathcal{M}_i^G$ is an embedding defined on all of $\A$, then $f^G$ is an isomorphism.
\end{lemma}

We now prove the following lemma on challenges issued by $P$-strategies during the construction.

\begin{lemma}
    An $S$-strategy $\alpha$ will be challenged by at most one $P$-strategy at any given stage.
\end{lemma}
\begin{proof}
    Let $\alpha$ be an $S$-strategy and let $\beta$ be a $P$-strategy such that $\beta\supseteq\alpha^\frown\<\infty\>$. If $\beta$ challenges $\alpha$ at some stage $s$, then it takes the $w_2$ outcome for the first time. The $P$-strategies extending $\beta^\frown\<w_2\>$ will then wait as in \textit{Case 1}, and so will not challenge $\alpha$. Until $\alpha$ can meet its challenge, it will continue to take the $w_{n_\beta}$ outcome, and so $P$-strategies extending $\alpha^\frown\<w_{n_\beta}\>$ will not be able to challenge $\alpha$ since $w_{n_\beta}\neq\infty$.
\end{proof}

The next two lemmas will prove facts about our $1$-generic $G$ that will be useful for our $S$-strategies.

\begin{lemma}\label{lemma: G is Delta_2}
    $G$ is $\Delta_2^0$.
\end{lemma}
\begin{proof}
    Let $m$ be arbitrary and we will show that $G$ can only change finitely often below $m$. The only strategies which can change $G$ are the $R$- and $P$-strategies. If $\alpha$ is a $P$-strategy, then $\alpha$ changes $G$ below $m$ if and only if its parameter $n_\alpha\leq m$. If $\alpha$ is instead an $R$-strategy, then $\alpha$ changes $G$ below $m$ if and only if there is an $S$-strategy $\beta^\frown\<\infty\>\subseteq\alpha$ such that $u_{\beta,n_\alpha}\leq m$. In either case, $\alpha$ will only act once to change $G$ according to its strategy unless it is initialized and chooses new large parameters. Since parameters are always chosen large (and so are never reused), only finitely many strategies can change $G$ below $m$.
\end{proof}

\begin{lemma}\label{lemma: true copies of components remain}
    If $\A\cong\mathcal{M}_i^G$, then for each $n$, there is an $s$ such that for all $t\geq s$, the true copies of the $2n$th and $(2n+1)$st components of $\A$ in $\mathcal{M}_i^G$ are the oldest and lexicographically least isomorphic copies in $\mathcal{M}_i^G[t]$ of these components.
\end{lemma}
\begin{proof}
    By Lemma \ref{lemma: G is Delta_2}, we have that $G$ is $\Delta_2^0$. If $\A\cong\mathcal{M}_i^G$, then for each $n$, we have true copies of the $2n$th and $(2n+1)$st components of $\A$ in $\mathcal{M}_i^G$. Suppose that the associated $G$-uses for the true copies of the $2n$th and $(2n+1)$st components are $u_{2n}$ and $u_{2n+1}$, respectively. Since $G$ is $\Delta_2^0$, there exists a sufficiently large stage $s_n$ such that $G[s_n]\restriction u_{2n}=G\restriction u_{2n}$, and so after stage $s_n$, the copy of the $2n$th component in $\A$ will become the oldest and lexicographically least isomorphic copy of the component in $\mathcal{M}_i^G$. The same holds true for the $(2n+1)$st component of $\A$.
\end{proof}





\subsubsection{Lemmas for each strategy}
For each strategy in our construction, we prove several key lemmas. We begin with the $P$- and $R$-strategies.

\begin{lemma}\label{lemma: waiting extends upwards}
    Let $\alpha$ be a $P$- or $R$-strategy. If $\beta$ is another $P$- or $R$-strategy such that $\beta^\frown\<w_0\>\subseteq\alpha$, then $\alpha$ takes the $w_0$ outcome at every $\alpha$-stage.
\end{lemma}
\begin{proof}
    If $\beta^\frown\<w_0\>\subseteq\alpha$ and $s$ is an $\alpha$-stage, then $n_\beta$ is not defined at stage $s$. Therefore, $\alpha$ doesn't define $n_\alpha$ at stage $s$, and so it takes the $w_0$ outcome.
\end{proof}

Once an $R$-strategy $\alpha$ defines its parameter, it will stay defined for the rest of the construction unless $\alpha$ is initialized. We show that if $\alpha$ is on the true path, then it will define an initial segment of $G$ that either meets or avoids its given c.e.\ set. We first begin with a short lemma.

\begin{lemma}\label{lemma: G does not change under n alpha}
    Let $\alpha$ be an $R$- or $P$-strategy that acts in \textit{Case 3} at stage $s$. Then $G[s]\restriction n_\alpha=G[s-1]\restriction n_\alpha$.
\end{lemma}
\begin{proof}
    If $\alpha$ is an $R$-strategy, this follows immediately because $G[s-1]\restriction n_\alpha\subseteq\tau$. If $\alpha$ is a $P$-strategy, then $\alpha$ may enumerate a use $u_{\beta,n_\alpha}$ into $G$ where $\beta$ is an $S$-strategy where $\beta^\frown\<\infty\>\subseteq\alpha$. In this case, since $u_{\beta,n_\alpha}$ was chosen large, we have that $u_{\beta,n_\alpha}>n_\alpha$ and so the only change to $G$ occurs above $n_\alpha$.
\end{proof}

\begin{lemma}\label{lemma: R parameters stay defined}
    Let $\alpha$ be an $R_j$-strategy that defines $n_\alpha$ at a stage $s_0$.
    \begin{enumerate}
        \item Unless $\alpha$ is initialized, at all stages $s>s_0$, we have that $G[s]\restriction n_\alpha=G[s_0]\restriction n_\alpha$.
        \item Suppose that $\alpha$ acts as in \textit{Case 3} at a stage $s_1>s_0$ with the string $\tau$ where $\tau\in W_j[s_1]$. Unless $\alpha$ is initialized, for all stages $s>s_1$, we have that $G[s]\restriction|\tau|=\tau$.
    \end{enumerate}
\end{lemma}
\begin{proof}
    Let $\alpha$ and $n_\alpha$ be as above. For (1), by Lemma \ref{lemma: G does not change under n alpha} all $R$- or $P$-strategies $\beta$ where $\beta\supset\alpha$ cannot change $G[s]$ below $n_\alpha$, and hence cannot change $G$ below $n_\alpha$. Therefore, if there is a change to $G[s]$ below $n_\alpha$, it was caused by some $R$- or $P$-strategy $\delta$ acting in \textit{Case 3} with $\delta^\frown\<w_1\>\subseteq\alpha$. But when $\delta$ acts, it takes the $w_2$ outcome, which initializes $\alpha$.

    For (2), we have that $\alpha$ takes the $w_2$ outcome at stage $s_1$ and the $s$ outcome after stage $s_1$. All $\beta\supseteq\alpha^\frown\<w_2\>$ or $\beta\supseteq\alpha^\frown\<s\>$ will define large parameters $n_\beta>|\tau|$. So if there are any changes to $G$ below $|\tau|$, it is caused by some $R$- or $P$-strategy $\delta^\frown\<w_1\>\subseteq\alpha$. However, if $\delta$ changes $G$, then it takes the $w_2$ outcome, but this initializes $\alpha$.
\end{proof}

We prove something similar for the $S$-strategies.

\begin{lemma}
    Let $\alpha$ be an $S$-strategy and let $s_0<s_1$ be $\alpha$-stages such that $n_\alpha[s]$ has already been defined and $\alpha$ is not initialized between $s_0$ and $s_1$. Then 
    \[
    G[s_0]\restriction n_\alpha[s_0]=G[s_1]\restriction n_\alpha[s_0]
    \]
    unless some $R$- or $P$-strategy $\beta\supseteq\alpha^\frown\<\infty\>$ with $n_\beta<n_\alpha[s_0]$ acts between stages $s_0$ and $s_1$.
\end{lemma}
\begin{proof}
    By Lemma \ref{lemma: R parameters stay defined}, $G$ can change below $n_\alpha[s_0]$ only if an $R$- or $P$-strategy $\beta$ with $n_\beta<n_\alpha[s_0]$ acts. Therefore, we cannot have that $\alpha^\frown\<\infty\><_L\beta$. We also cannot have that $\beta<_L\alpha$ or $\beta\subseteq\alpha$ since $\alpha$ would get initialized. Hence, it must be that $\alpha^\frown\<\infty\>\subseteq\beta$.
\end{proof}

We now prove that if a $P_e$-strategy $\alpha$ takes action to add diagonalizing loops to a pair of $A$-components, then the diagonalization against $\Phi_e$ on these components will remain at the end of the construction.

\begin{lemma}\label{lemma: P diagonalization works as intended}
    Let $\alpha$ be a $P$-strategy. Suppose that $\alpha$ adds diagonalizing loops to the $2n_\alpha$th and $(2n_\alpha+1)$st components of $\A$ at a stage $s_0$. Unless $\alpha$ is initialized, these components are not homogenized at a future stage.
\end{lemma}
\begin{proof}
    Let $\alpha$ be as in the statement of the lemma. We only homogenized these components if the path moves to the left of $\alpha$, in which case $\alpha$ is initialized.
\end{proof}

We now prove several key facts for the $S$-strategies that will help with the main verification lemma.

\begin{lemma}\label{lemma: P strategy - lower bounds}
    Let $\alpha$ be an $S$-strategy and $\beta$ be a $P$-strategy such that $\alpha^\frown\<\infty\>\subseteq\beta$. Suppose that there are stages $t_0<t_1<t_2<t_3$ such that $\beta$ defines $m_\beta$ at $t_0$, $\beta$ defines $n_\beta$ at $t_1$, $\beta$ acts in \textit{Case 3} at $t_2$, $t_3$ is the first $\alpha$-stage after $t_2$, and $\beta$ is not initialized between $t_0$ and $t_3$. Then,
    \begin{enumerate}
        \item$n_\alpha[t_1]<n_\beta[t_1]$, and for all $S$-strategies $\gamma$ such that $\gamma^\frown\<\infty\>\subseteq\beta$, $f_\gamma^G$ is defined on the $2m_\beta$th and $(2m_\beta+1)$st components of $\A$ with $n_\beta[t_1]>u_{\gamma,m_\beta}[t_1]$;
        \item for $t_1\leq t<t_2$, $n_\alpha[t]>m_\beta[t]=m_\beta[t_0]$; and
        \item $n_\alpha[t_3]>m_\beta[t_3]=m_\beta[t_0]$.
    \end{enumerate}
\end{lemma}
\begin{proof}
    Note that the values of $m_\beta$ and $n_\beta$ do not change once they are defined since $\beta$ is not initialized. For ($1$), $\beta$ defines $n_\beta$ large at stage $t_1$, and so $n_\beta[t_1]>n_\alpha[t_1]$. In addition, $\beta$ does not define $n_\beta$ until all $S$-strategies $\gamma^\frown\<\infty\>\subseteq\beta$ have $n_\gamma>m_\beta$. Therefore, when $n_\beta$ is chosen large at stage $t_1$, we have that $n_\beta[t_1]>u_{\gamma,m_\beta}[t_1]$.

    For ($2$), when $\beta$ defines $n_\beta$ at stage $t_1$, $f_\alpha^G$ is defined on the $2m_\beta$th and $(2m_\beta+1)$st components of $\A$. Therefore, we have that $n_\alpha[t_1]>m_\beta[t_0]$. At stage $t_1$, $\beta$ takes the $w_1$ outcome for the first time, and so all strategies $\gamma$ where $\beta^\frown\<w_1\>\subseteq\gamma$ choose parameters larger than $n_\alpha[t_1]$. In particular, none of these strategies can lower the value of $n_\alpha$ below $m_\beta$. The only other strategies which can lower $n_\alpha$ without initializing $\alpha$ are $\gamma$ where $\alpha^\frown\<\infty\>\subseteq\gamma\subset\beta$. However, these $\gamma$ would initialize $\beta$. Hence, we have that for all $t_1\leq t<t_2$ that $n_\alpha[t]>m_\beta$. In addition, at stage $t_2$, $\beta$ acts after $\alpha$, so $n_\alpha$ does not get redefined to reflect $\beta$'s action until the $\alpha$-stage $t_3$. It follows that $n_\alpha[t]>m_\beta$ for $t_2\leq t<t_3$.

    For ($3$), at stage $t_3$, $\alpha$ sets $n_\alpha[t_3]$ to be the least $m$ such that $f_\alpha^G[t_3]$ is not defined on the $2m$th and $(2m+1)$st components of $\A$. At stage $t_2$, $\beta$ enumerated $u_{\gamma,n_\beta}$ for $S$-strategies $\gamma^\frown\<\infty\>\subseteq\beta$. However, $u_{\gamma,n_\beta}>n_\beta$ because $u_{\gamma,n_\beta}$ was defined large when it was chosen and $n_\beta>u_{\alpha,m_\beta}$ by ($1$). Therefore, the computation $f_\alpha^G$ on the $2m_\beta$th and $(2m_\beta+1)$st components is not destroyed by $\beta$'s action at stage $t_2$. It follows that $n_\alpha[t_3]>m_\beta[t_3]=m_\beta[t_0]$. 
\end{proof}

\begin{lemma}\label{lemma: R strategy - lower bounds}
    Let $\alpha$ be an $S$-strategy and $\beta$ be an $R$-strategy such that $\alpha^\frown\<\infty\>\subseteq\beta$. Suppose that there are stages $t_0<t_1<t_2<t_3$ such that $\beta$ defines $m_\beta$ at $t_0$, $\beta$ defines $n_\beta$ at $t_1$, $\beta$ acts in \textit{Case 3} at $t_2$, $t_3$ is the next $\alpha$-stage after $t_2$, and $\beta$ is not initialized between $t_0$ and $t_3$. Then,
    \begin{enumerate}
        \item $m_\beta<n_\alpha[t_1]<n_\beta[t_1]$ and $n_\beta[t_1]>u_{\alpha,m_\beta}[t_1]$, which is the use of $f_\alpha^G[t_1]$ on the $2m_\beta$th and $(2m_\beta+1)$st components of $\A$;
        \item for all $t_1\leq t<t_3$, $n_\alpha[t]>m_\beta$; and
        \item $n_\alpha[t_3]>m_\beta$.
    \end{enumerate}
\end{lemma}
\begin{proof}
    ($1$) and ($2$) hold as in Lemma \ref{lemma: P strategy - lower bounds}. For ($3$), when $\beta$ acts at stage $t_2$, it defines $G[t_2]=\tau^\frown0^\omega$ where $G[t_2-1]\restriction n_\beta\subseteq\tau$. Therefore, if $u_{\alpha,m_\beta}[t_1]$ is the use of the map $f_\alpha^G[t_1]$ on the $2m_\beta$th and $(2m_\beta+1)$st components, then 
    \[
    G[t_2]\restriction u_{\alpha,m_\beta}=G[t_2-1]\restriction u_{\alpha,m_\beta}.
    \]
    Moreover, no requirements can change $G$ below $n_\beta$ between stages $t_2$ and $t_3$ without initializing $\beta$. Hence,
    \[
    G[t_3-1]\restriction u_{\alpha,m_\beta}=G[t_2-1]\restriction u_{\alpha,m_\beta}.
    \]
    In particular, when $\alpha$ acts in \textit{Case 2} at stage $t_3$, either $\alpha$ acts in \textit{Subcase 1} and doesn't change $n_\alpha$, or it acts in \textit{Subcase 3} and we retain $n_\alpha[s_3]>m_\beta$ because $G$ has not changed below $u_{\alpha,m_\beta}$. 
\end{proof}

We now combine Lemmas \ref{lemma: P strategy - lower bounds} and \ref{lemma: R strategy - lower bounds} into the following single lemma.

\begin{lemma}\label{lemma: lower bounds for S parameters}
    Let $\alpha$ be an $S$-strategy. An $R$- or $P$-strategy $\beta\supseteq\alpha^\frown\<\infty\>$ cannot cause $n_\alpha$ to drop below $m_\beta$.
\end{lemma}

With the help of Lemma \ref{lemma: lower bounds for S parameters}, we now prove the following important fact about the $S$-strategies.

\begin{lemma}\label{lemma: S parameters are unbounded}
    Let $\alpha$ be an $S$-strategy that is never initialized after stage $s$ and takes the $\infty$ outcome infinitely often. For all $m$, there exists a stage $t_m$ such that $n_\alpha[t]\geq m$ for all $t\geq t_m$.
\end{lemma}
\begin{proof}
    Fix $m$. Let $\beta_0,\dots,\beta_k$ be the $R$- and $P$-strategies such that $\alpha^\frown\<\infty\>\subseteq\beta_i$ and $m_{\beta_i}$ is defined at some stage where $m_{\beta_i}<m$. Let $t>s$ be a stage such that no $\beta_i$ acts as in \textit{Case 3} to change $G$ with $m_{\beta_i}<m$ after stage $t$. After stage $t$, no requirement can cause $n_\alpha$ to drop below $m$. Moreover, no other $R$- or $P$-requirement $\delta$ will act until $n_\alpha>m_\delta>m$. Let $t_m>t$ be the least stage such that $n_\alpha[t_m]>m$. 
\end{proof}

\subsubsection{Lemmas on interactions between multiple strategies}

In this section, we prove several lemmas that detail how tension between multiple strategies are resolved in our construction. We first show that for a lower priority $R$-strategy, it won't be able to injure higher priority $S$-strategies at arbitrarily small numbers.

\begin{lemma}\label{lemma: waiting procedure works}
    Let $\alpha$ be an $R_j$-strategy that acts in \textit{Case 3} at stage $s$. For all $P$- and $R$-strategies $\beta\subset\alpha$ and all $S$-strategies $\gamma$ such that $\gamma^\frown\<\infty\>\subseteq\alpha$, if $f_\gamma^G$ is defined on the $2n_\beta$th and $(2n_\beta+1)$st components of $\A$, then
    \[
    G[s]\restriction u_{\gamma,n_\beta}=G[s-1]\restriction u_{\gamma,n_\beta}.
    \]
    Therefore, these maps remain defined.
\end{lemma}
\begin{proof}
   When $\alpha$ first acts, it waits for $n_\beta$ to be defined for all $\beta\subset\alpha$ and then defines $m_\alpha=\max\{n_\beta : \beta\subset\alpha\}$ and does not define $n_\alpha$ until it sees that for all $S$-strategies $\gamma^\frown\<\infty\>\subseteq\alpha$, the map $f_\gamma^G$ has been defined on the $2m_\alpha$th and $(2m_\alpha+1)$st components of $\A$. When $\alpha$ is finally able to define $n_\alpha$, it picks $n_\alpha$ large and so $n_\alpha>u_{\gamma,n_\beta}$ for all $S$-strategies $\gamma^\frown\<\infty\>\subseteq\alpha$ and for all $P$- and $R$-strategies $\beta\subset\alpha$. If $\alpha$ finds a $\tau$ such that $\tau\in W_j[s]$ and $G[s-1]\restriction n_\alpha\subseteq\tau$ at a stage $s$, it sets $G[s]=\tau^\frown0^\omega$. This will not change $G$ below $u_{\gamma,n_\beta}$ for all $\gamma$ and $\beta$ as above since $n_\alpha>u_{\gamma,n_\beta}$.
\end{proof}

The following lemma details how the $S$-strategies are able to undo any of their previously defined maps on $\A$-components in the event that a pair of components were used to diagonalize against a computable $\Phi_e$ by some $P_e$-strategy.

\begin{lemma}\label{lemma: uses of loops are smaller than uses of maps}
     Let $\alpha$ be an $S_i$-strategy. If $f_\alpha^G[s]$ is defined on the $2m$th and $(2m+1)$st components of $\A$ at stage $s$, then $l_m[s]<u_{\alpha,m}[s]$ where $l_m[s]$ is the max use of the edges in the image of these components in $\mathcal{M}_i^G$ and $u_{\alpha,m}[s]$ is the use of this computation. In particular, any change in these loops in $\mathcal{M}_i^G$ causes the computation to be destroyed.
\end{lemma}
\begin{proof}
    Whenever an $f_\alpha^G[s]$ computation is defined, its use is defined large. Therefore, we have that $l_m[s]<u_{\alpha,m}[s]$ when the computation is first defined. The computation may be destroyed later by a change in $G$ below $u_{\alpha,m}[s]$. However, if it is later restored by $u_{\alpha,m}[t]=u_{\alpha,m}[s]$ for $t>s$, then $l_m[t]=l_m[s]$ and so the loops in $\mathcal{M}_i^G$ are restored as well.
\end{proof}

We now show that enumerating these uses has the intended effect of deleting a map $f_\alpha^G$ defined previously by some $S$-strategy $\alpha$ on some $\A$-components.

\begin{lemma}\label{lemma: uses were never in G before}
    Let $\alpha$ be a $P$-strategy and $\beta$ be an $S$-strategy such that $\beta^\frown\<\infty\>\subseteq\alpha$. Suppose $s_0<s_1<s_2$ are stages such that $\alpha$ defines $n_\alpha$ at $s_0$, $\beta$ defined $f_\beta^G$ on the $2n_\alpha$th and $(2n_\alpha+1)$st components in $\A$ at $s_1$, $\alpha$ acts in \textit{Case 3} at $s_2$ and puts $u_{\beta,n_\alpha}$ into $G[s_2]$, and $\alpha$ is not initialized between $s_0$ and $s_2$. Then,
    \begin{enumerate}
        \item $n_\alpha[s_0]>n_\beta[s_0]$,
        \item for all $s_1\leq t<s_2$, we have that the values of $n_\alpha[t]=n_\alpha[s_0]$ and $u_{\beta,n_\alpha}[t]=u_{\beta,n_\alpha}[s_1]$, so we denote them by $n_\alpha$ and $u_{\beta,n_\alpha}$,
        \item for all $t<s_2$, $u_{\beta,n_\alpha}\not\in G[t]$ and therefore $G[s_2]\restriction(u_{\beta,n_\alpha}+1)\neq G[t]\restriction(u_{\beta,n_\alpha}+1)$ for all $t<s_2$,
        \item unless $\alpha$ is initialized, $u_{\beta,n_\alpha}\in G[s]$ for all $s\geq s_2$, and
        \item if there is an $s\geq s_2$ such that $u_{\beta,n_\alpha}\not\in G[s]$, then the $2n_\alpha$th and $(2n_\alpha+1)$st components of $\A$ and $\B$ are homogenized.
    \end{enumerate}
\end{lemma}
\begin{proof}
    ($1$) follows from the fact that $n_\alpha$ is defined large at stage $s_0$. For ($2$), $n_\alpha$ changes values only when $\alpha$ is initialized, and the only strategies that can change $G$ below $u_{\beta,n_\alpha}$ would initialize $\alpha$. For ($3$), when $\beta$ defines $u_{\beta,n_\alpha}$ at stage $s_1$, the use was chosen large so $u_{\beta,n_\alpha}\not\in G[t]$ for $t\leq s_1$. For $s_1<t<s_2$, the only strategies that could change $G$ below $u_{\beta,n_\alpha}$ would initialize $\alpha$ by Lemma \ref{lemma: waiting procedure works}. Lastly for ($4$), only an $R$-strategy $\gamma$ could remove $u_{\beta,n_\alpha}$ and only if $n_\gamma<u_{\beta,n_\alpha}$. But such an $R$-strategy with $n_\gamma<u_{\beta,n_\alpha}$ would initialize $\alpha$ when it acts by Lemma \ref{lemma: waiting procedure works}. Furthermore, the path $\pi_s<_L\alpha$ causes the $2n_\alpha$th and $(2n_\alpha+1)$st components of $\A$ and $\B$ to be homogenized, proving ($5$).
\end{proof}

\begin{lemma}\label{lemma: 2.3.18}
    Let $\alpha$ be a $P$-strategy and $\beta$ be an $S$-strategy such that $\beta^\frown\<\infty\>\subseteq\alpha$. Suppose $\alpha$ acts at stage $t$ in \textit{Case 3} and puts $u_{\beta,n_\alpha}$ into $G$ and challenges $\beta$. Unless $\beta$ is initialized, at the next $\beta$-stage $s$, $\beta$ defines $n_\beta[s]\leq n_\alpha[s]=n_\alpha[t]$.
\end{lemma}
\begin{proof}
    By Lemma \ref{lemma: uses were never in G before} and the fact that $\beta$ is not initialized, we get that $u_{\beta,n_\alpha}\in G[s-1]$. Therefore, $f_\beta^G[s-1]$ is no longer defined on the $2n_\alpha$th and $(2n_\alpha+1)$st components of $\A$ and hence $n_\beta$ is redefined to a value which is at most $n_\alpha$.
\end{proof}

We now state and prove the main verification lemma for our construction.

\begin{lemma}[Main Verification Lemma]\label{lemma: main verification lemma}
    Let $\pi=\liminf_s \pi_s$ be the true path of the construction, where $\pi_s$ denotes the current true path at stage $s$ of the construction. Let $\alpha\subset \pi$.
    \begin{enumerate}
        \item If $\alpha$ is an $R_j$-strategy, then the parameters $m_\alpha$ and $n_\alpha$ are eventually permanently defined and there is an outcome $o$ and an $\alpha$-stage $t_\alpha$ such that for all $\alpha$-stages $s\geq t_\alpha$, $\alpha$ takes outcome $o$ where $o$ ranges over $\{s,w_1\}$. 
        \item Let $\alpha$ be a $P_e$-strategy, then the parameters $m_\alpha$ and $n_\alpha$ are eventually permanently defined and there is an outcome $o$ and an $\alpha$-stage $t_\alpha$ such that for all $\alpha$-stages $s\geq t_\alpha$, $\alpha$ takes outcome $o$ where $o$ ranges over $\{s,w_1\}$.
        \item If $\alpha$ is an $S_i$-strategy, then either $\alpha$ takes outcome $\infty$ infinitely often or there is an outcome $w_n$ and a stage $t_\alpha$ such that for all $\alpha$-stages $s>t_\alpha$, $\alpha$ takes outcome $w_n$. If $\A\cong\mathcal{M}_i^G$, then $\alpha$ takes the $\infty$ outcome infinitely often and defines an embedding $f_\alpha^G:\A\to\mathcal{M}_i^G$ which can be extended in a $G$-computable way to a $G$-computable isomorphism $\hat{f}_\alpha^G$ between $\A$ and $\mathcal{M}_i^G$.
    \end{enumerate}
    In addition, $\alpha$ satisfies its assigned requirement.
\end{lemma}
\begin{proof}
    For ($1$), let $\alpha\subset \pi$ be an $R_j$-strategy and let $s_0$ be the least stage such that $\alpha\leq_L \pi_s$ for all $s\geq s_0$ and for all $R$- and $P$-strategies $\beta\subset\alpha$, $n_\beta$ is defined. At stage $s_0$, $\alpha$ defines $m_\alpha$ permanently since $\alpha$ is never initialized again. By Lemma \ref{lemma: S parameters are unbounded}, there is an $s_1>s_0$ such that for all $S$-strategies $\beta$ where $\beta^\frown\<\infty\>\subseteq\alpha$, we have $n_\beta[t]>m_\alpha$ for all $t\geq s_1$. At stage $s_1$ (if not before), $\alpha$ defines $n_\alpha$ permanently and takes outcome $w_1$.
    
    If $\alpha$ remains in the first part of \textit{Case 3} from its description for all $\alpha$-stages $s\geq s_0$, then $R_j$ is satisfied because for all extensions $\tau$ of $\sigma=G[s_0]\restriction n_\alpha$, $\tau\not\in W_j$. Moreover, since $\alpha\subset \pi$, we have that $G[s_0]\restriction n_\alpha=G\restriction n_\alpha$ by Lemma \ref{lemma: R parameters stay defined}. Otherwise, there is a stage $s\geq s_0$ such that $\alpha$ finds an extension $\tau\supseteq G[s_0]\restriction n_\alpha$ where $\tau\in W_j[s]$. At the end of stage $s$, $\alpha$ takes the $w_2$ outcome, and since $\alpha\subset \pi$, $\alpha$ will be able to act again and finally take the $s$ outcome. For the rest of the construction, it will be in \textit{Case 5} of its description after defining $\tau\subseteq G$ where $\tau\in W_j$. Again by Lemma \ref{lemma: R parameters stay defined}, we have that $G[t]\restriction|\tau|=\tau$ for all $t\geq s$, and $R_j$ is satisfied.

    We now prove ($2$). Suppose $\alpha\subset \pi$ is a $P_e$-strategy. Let $s_0$ be the least stage such that $\alpha\leq_L \pi_s$ for all $s\geq s_0$ and for all $R$- or $P$-strategies $\beta\subset\alpha$, $n_\beta$ is defined. Like before, at stage $s_0$, $\alpha$ defines $m_\alpha$ permanently since $\alpha$ is never initialized. By Lemma \ref{lemma: S parameters are unbounded}, there exists a stage $s_1>s_0$ such that for all $S$-strategies $\beta$ with $\beta^\frown\<\infty\>\subseteq\alpha$, we have that $n_\beta[t]>m_\alpha$ for all $t\geq s_1$. Then at stage $s_1$ (if not before), $\alpha$ defines $n_\alpha$ permanently and takes outcome $w_1$.
    
    If $\alpha$ remains in the first part of \textit{Case 3} from its description for all $\alpha$-stages $s\geq s_0$, then $P_e$ is trivially satisfied since $\Phi_e$ is not total or maps $\A$ incorrectly into $\B$, and so it cannot be an isomorphism between $\A$ and $\B$. In this case, $\alpha$ takes outcome $w_1$ cofinitely often.

    Otherwise, there is an $\alpha$-stage $s_1>s_0$ where $\Phi_e[s_1]$ maps the $2n$th and $(2n+1)$st components of $\A$ isomorphically into $\B$. Then, $\alpha$ carries out all actions described in the second part of \textit{Case 3}. Let $s_2>s_1$ be the next $\alpha$-stage, and now $\alpha$ is in \textit{Case 4} of its description and can now take the $s$ outcome. Moreover, since $\alpha\subset \pi$, $\alpha$ will never get initialized again and so the $2n$th and $(2n+1)$st components of $\A$ and $\B$ are never homogenized. By Lemma \ref{lemma: P diagonalization works as intended} we have that $\Phi_e(a_{2n})=b_{2n}$, but $a_{2n}$ is connected to a cycle of length $5n+3$ whereas $b_{2n}$ is connected to a cycle of length $5n+4$. Similarly, $\Phi_e(a_{2n+1})=b_{2n+1}$, but $a_{2n+1}$ is connected to a cycle of length $5n+4$ whereas $b_{2n+1}$ is connected to a cycle of length $5n+3$. Hence, $\Phi_e$ cannot be an isomorphism between $\A$ and $\B$, and so $P_e$ is satisfied.

    For ($3$), let $\alpha\subset \pi$ be an $S_i$-strategy and let $s_0$ be the least stage such that for all stages $s\geq s_0$, $\alpha\leq_L \pi_s$. Under all of the subcases of \textit{Case 3}, $\alpha$ enacts the module to search for extensions of $f_\alpha^G$ on the $2n_\alpha$th and $(2n_\alpha+1)$st components of $\mathcal{A}$ for its currently defined $n_\alpha$ parameter. If an extension cannot be found for these components, then we have that $\alpha$ takes the $w_{n_\alpha}$ outcome at cofinitely many $\alpha$-stages after $s_0$ and we are done since $\A\not\cong\mathcal{M}_i^G$. 
    
    We can then assume that $\A\cong\mathcal{M}_i^G$ and that $\alpha$ takes the $\infty$ outcome infinitely often. $\alpha$ can then eventually find a correct extension on these components by Lemma \ref{lemma: true copies of components remain}. Additionally, we have by Lemma \ref{lemma: S parameters are unbounded} that $n_\alpha[s]\to\infty$ as $s\to\infty$ where $n_\alpha[s]$ denotes the value of $n_\alpha$ at the end of stage $s$. We also have that by Lemmas \ref{lemma: uses of loops are smaller than uses of maps} and \ref{lemma: 2.3.18}, these embeddings will remain or be able to recover if they are challenged by lower priority strategies.

    Although $f_\alpha^G$ is not defined on the homogenizing loops added at the end of each stage, we can $G$-computably extend $f_\alpha^G$ to a map $\hat{f}_\alpha^G$ defined on all of $\mathcal{A}$ in the following way. Since $\mathcal{A}\cong\mathcal{M}_i^G$, then these loops added at the end of each stage will eventually have true copies in the images under $f_\alpha^G$ on the affected components, and they will become the oldest and lex-least such copies by Lemma \ref{lemma: true copies of components remain}. We can find copies of these new loops by using $G$ since $\mathcal{M}_i^G$ is $G$-computable, and once we know when these loops appear, we can extend $f_\alpha^G$ appropriately on these new components to obtain $\hat{f}_\alpha^G$. Our new map $\hat{f}_\alpha^G$ is still $G$-computable, and by Lemma \ref{lemma: graph isomorphisms}, $\hat{f}_\alpha^G$ is our $G$-computable isomorphism between $\mathcal{A}$ and $\mathcal{M}_i^G$.
\end{proof}

\section{Other classes of structures}\label{section: other classes of structures}
For this section, we restate the main result from \cite{villano2024computable}.

\begin{theorem}\label{thm: poset result}
    Let $P=(P,\leq)$ be a computable partially ordered set and let $P=P_0\sqcup P_1$ be a computable partition. Then, there exists a computable computably categorical directed graph $\G$ and an embedding $h$ of $P$ into the c.e.\ degrees where $\G$ is computably categorical relative to each degree in $h(P_0)$ and is not computably categorical relative to each degree in $h(P_1)$.
\end{theorem}

We now show that the structure which witnesses the above behavior need not be a directed graph.

\begin{theorem}\label{thm: in other classes of structures}
    Let $P=(P,\leq)$ be a computable partially ordered set and let $P=P_0\sqcup P_1$ be a computable partition. For the following classes of structures, there exists a computable example in each class which satisfies the conclusion of Theorem \ref{thm: poset result}: symmetric, irreflexive graphs; partial orderings; lattices; rings with zero-divisors; integral domains of arbitrary characteristic; commutative semigroups; and $2$-step nilpotent groups. 
\end{theorem}

We will use the codings given by the authors in \cite{HKSS02} in order to code the directed graph $\G$ (and presentations of it) in the statement of Theorem \ref{thm: poset result} into a structure in one of the above classes of structures. 

\subsection{Codings}
Suppose an abstract graph $\G$ and a structure $\A$, from one of the listed classes in the statement of Theorem \ref{thm: in other classes of structures}, have computable presentations. Let $G$ and $A$ be particular copies of $\G$ and $\A$, respectively. We first state the following definitions before defining the coding from \cite{HKSS02}.

\begin{definition}
    A relation $U$ on a structure $\A$ is \textit{invariant} if for every automorphism $f:\A\cong\A$, we have that $f(U)=U$.
\end{definition}

Here, an $n$-ary relation $U$ on a structure $\A$ is some subset of $|\A|^n$ where $|\A|$ denotes the underlying domain of $\A$.

\begin{definition}
    A relation $U$ on the domain of a structure $\A$ is \textit{intrinsically computable} if for any computable $\mathcal{B}$ and computable isomorphism $f:\mathcal{A}\to\mathcal{B}$, the image $f(U)$ is computable.
\end{definition}

\begin{definition}
    Let $\mathbf{d}$ be a degree. A $\mathbf{d}$-\textit{computable defining family} for a structure $\A$ is a $\mathbf{d}$-computable set of existential formulas $\phi_0(\vec{a},x),\phi_1(\vec{a},x),\dots$ such that $\vec{a}$ is a tuple of elements of $|\A|$, each $x\in|\A|$ satisfies some $\phi_i$, and no two elements of $|\A|$ satisfy the same $\phi_i$.
\end{definition}

The main idea of the coding methods given in \cite{HKSS02} is that there are intrinsically computable, invariant relations $D(x)$ and $R(x,y)$ on the domain of $\A$ such that taking the elements in $|\A|$ satisfying $D(x)$ and adding the relation on them defined by $R(x,y)$ gives a copy of the graph $\G$. This gives us a map from copies of the structure $\A$ to copies of $\G$, which we will write as $A\mapsto G_A$. In addition, there is a uniform computable functional taking copies of $\G$ to copies of $\A$, which we will write as $G\mapsto A_G$. Note that these maps can be applied repeatedly. For example, we can apply a map to go from $A$ to $G_A$ and then to $A_{G_A}$, a copy of the structure $\A$. Note that $A_{G_A}$ and $A$ are isomorphic as both are copies of the structure $\A$, but they are \textit{not} the same presentation. 

The coding methods in \cite{HKSS02} satisfy the following list of properties:
\begin{enumerate}
    \item[(P0)] For every presentation $G$ of $\G$, the structure $A_G$ is $\deg(G)$-computable.
    \item[(P1)] For every presentation $G$ of $\G$, there is a $\deg(G)$-computable map $g_G:|A_G|\to G$ such that $R^{A_G}(x,y)\iff E^G(g_G(x),g_G(y))$ for all $x,y\in |A_G|$.
    \item[(P2)] If $f:|\A|\to |\A|$ is $1$-to-$1$ and onto and $R(x,y)\iff R(f(x),f(y))$ for all $x,y\in |\A|$, then $f$ can be extended to an automorphism of $\A$.
    \item[(P3)] For every presentation $G$ of $\G$, there is a $\deg(G)$-computable set of existential formulas $\phi_0(\vec{a},\vec{b}_0,x), \phi_1(\vec{a},\vec{b}_1,x),\dots$ such that $\vec{a}$ is a tuple of elements from the universe of $A_G$, each $\vec{b}_i$ is a tuple of elements of $|A_G|$, each $x$ in the universe of $A_G$ satisfies some $\phi_i$, and no two elements of the universe of $A_G$ satisfy the same $\phi_i$.
\end{enumerate}

Each of these properties are key in proving Lemmas $2.6$-$2.9$ in \cite{HKSS02}, whose relativized versions below are needed to prove Theorem \ref{thm: in other classes of structures}. Here, we let $\G$ be the directed graph that satisfies the conclusion of Theorem \ref{thm: poset result}, but the lemmas hold in general for any directed graph and for all degrees $\mathbf{d}$.

\begin{lemma}\label{lemma: relativized lemma 1}
    For every $\mathbf{d}$-computable presentation $G$ of $\G$, there is a $\mathbf{d}$-computable isomorphism from $G_{A_G}$ onto $G$.
\end{lemma}

\begin{lemma}\label{lemma: relativized lemma 2}
    For every $\mathbf{d}$-computable presentation $A$ of $\A$, there is a $\mathbf{d}$-computable isomorphism from $A_{G_A}$ onto $A$.
\end{lemma}

\begin{lemma}\label{lemma: relativized lemma 3}
    If $G_0$ and $G_1$ are $\mathbf{d}$-computable presentations of $\G$ and $h:G_0\to G_1$ is a $\mathbf{d}$-computable isomorphism, then there is a $\mathbf{d}$-computable isomorphism $\hat{h}:A_{G_0}\to A_{G_1}$.
\end{lemma}

\begin{lemma}\label{lemma: relativized lemma 4}
    If $A_0$ and $A_1$ are $\mathbf{d}$-computable presentations of $\A$ and $h:A_0\to A_1$ is a $\mathbf{d}$-computable isomorphism, then there is a $\mathbf{d}$-computable isomorphism $\hat{h}:G_{A_0}\to G_{A_1}$.
\end{lemma}

We now prove a key lemma with the help of these relativized lemmas.

\begin{lemma}\label{lemma: coding preserves categoricity relative to a degree}
    For any degree $\mathbf{d}$, $\G$ is computably categorical relative to $\mathbf{d}$ if and only if $\A$ is computably categorical relative to $\mathbf{d}$.
\end{lemma}
\begin{proof}
    Fix the degree $\mathbf{d}$. Suppose $\G$ is computably categorical relative to $\mathbf{d}$. Let $A_0$ and $A_1$ be $\mathbf{d}$-computable presentations of $\A$. We then obtain $\mathbf{d}$-computable presentations $G_{A_0}$ and $G_{A_1}$ of $\G$. Since $\G$ is computably categorical relative to $\mathbf{d}$, there is a $\mathbf{d}$-computable isomorphism $h:G_{A_0}\to G_{A_1}$. By Lemma \ref{lemma: relativized lemma 3}, we have a $\mathbf{d}$-computable isomorphism $\hat{h}:A_{G_{A_0}}\to A_{G_{A_1}}$. Additionally, by Lemma \ref{lemma: relativized lemma 2}, there are $\mathbf{d}$-computable isomorphisms $f_0: A_{G_{A_0}}\to A_0$ and $f_1: A_{G_{A_1}}\to A_1$. It follows that $f_1\circ \hat{h}\circ f_0^{-1}:A_0\to A_1$ is a $\mathbf{d}$-computable isomorphism as needed.

    For the reverse direction, suppose $\A$ is computably categorical relative to $\mathbf{d}$. To show $\G$ is computably categorical relative to $\mathbf{d}$, we run an identical argument as above using Lemmas \ref{lemma: relativized lemma 1} and \ref{lemma: relativized lemma 4} in place of Lemmas \ref{lemma: relativized lemma 2} and \ref{lemma: relativized lemma 3}, respectively.
\end{proof}

Theorem \ref{thm: in other classes of structures} follows almost immediately.

\begin{proof}[Proof of Theorem \ref{thm: in other classes of structures}]
    Let $P$ be our computable poset where $P=P_0\sqcup P_1$ is a computable partition, let $\G$ be the computable directed graph which witnesses Theorem \ref{thm: poset result}, and let $h$ be the embedding of $P$ into the c.e.\ degrees. By Lemma \ref{lemma: coding preserves categoricity relative to a degree}, we have that since $\G$ is computably categorical, so is $\A$. We also have that because $\G$ is computably categorical relative to all $\mathbf{d}\in h(P_0)$ that $\A$ is computably categorical relative to all $\mathbf{d}\in h(P_0)$. Finally, since $\G$ is not computably categorical relative to any $\mathbf{d}\in h(P_1)$, it follows that $\A$ is also not computably categorical relative to any $\mathbf{d}\in h(P_1)$.
\end{proof}

\subsection{Linear orderings and Boolean algebras}
To end this section, we discuss two classes of structures in which the version of Theorem \ref{thm: in other classes of structures} where the structure is not computably categorical fails. That is, in both classes of structures, it is impossible to produce a computable structure $\mathcal{A}$ which is not computably categorical and a c.e.\ degree $\mathbf{d}$ such that $\mathcal{A}$ is computably categorical relative to $\mathbf{d}$. These observations illustrate the limitations of the techniques from \cite{villano2024computable} as a result of key structural facts about certain classes of structures, and not from the properties of certain degrees as with $2$-generics.

For computable linear orderings, Remmel in \cite{10.2307/2043534} proved the following.

\begin{theorem}[Remmel \cite{10.2307/2043534}]
    A computable linear order is computably categorical if and only if it has finitely many adjacent pairs.
\end{theorem}

Here, we say that for a linear ordering $(L,\leq)$, two elements $x,y$ are an \textit{adjacent pair} if either $x<y$ or $y<x$ and there is no $z\in L$ such that $x<z<y$ or $y<z<x$. Remmel's result also fully characterizes relative computable categoricity for computable linear orderings. So, if $L$ is a computable linear ordering which is computably categorical, then it must be computably categorical relative to all degrees $\mathbf{d}$, and so we cannot create a c.e.\ degree $\mathbf{d}$ such that $L$ is not computably categorical relative to $\mathbf{d}$.

Remmel's construction in the only if direction of his result also relativizes to any set, so for linear orderings which are not computably categorical, they must also fail to be categorical relative to all degrees $\mathbf{d}$. This differs greatly from what happens in the classes of structures listed in Theorem \ref{thm: in other classes of structures}, where computable examples could fail to be computably categorical but still be categorical relative to some c.e.\ degrees.

We now consider the class of computable Boolean algebras. For these structures, we have the following results.

\begin{theorem}[Goncharov \cite{Gon75b}]\label{thm: boolean algebra characterization of relatively c.c.}
    A computable Boolean algebra is computably categorical if and only if it has finitely many atoms.
\end{theorem}

This fully characterizes relative computable categoricity for computable Boolean algebras. Like with computable linear orderings, if a computable Boolean algebra $\B$ is computably categorical, then it must be computably categorical relative to all degrees $\mathbf{d}$. That is, no computable computably categorical Boolean algebra can witness the nonmonotonic behavior of computably categoricity relative to a degree. Additionally, Bazhenov showed the following.

\begin{theorem}[Bazhenov \cite{Bazhenov14}]\label{thm: d-categoricity of Boolean algebras}
    For every degree $\mathbf{d}<\mathbf{0}'$, a computable Boolean algebra is $\mathbf{d}$-computably categorical if and only if it is computably categorical.
\end{theorem}

By Bazhenov's result, if $\mathcal{B}$ is not computably categorical, then for all degrees $\mathbf{d}<\mathbf{0}'$, $\mathcal{B}$ is not $\mathbf{d}$-computably categorical. This implies that $\mathcal{B}$ is not computably categorical relative to any $\mathbf{d}<\mathbf{0}'$, and so Theorem \ref{thm: in other classes of structures} also fails for the class of computable Boolean algebras.

In light of the observations made above for linear orderings and Boolean algebras, we pose the following question.

\begin{question}
    For other classes of structures with an algebraic characterization of relative computable categoricity, does there exist an example of a structure which is not computably categorical but is computably categorical relative to some degree $\mathbf{d}>\mathbf{0}$?
\end{question}

The list of structures in this section is not exhaustive, and there are several examples in the literature in which computable categoricity and relativized computable categoricity coincide. One such example is the class of trees, in which all trees with finite height are computably categorical and relatively computably categorical \cite{Lempp05}. For those classes of structures, it would be interesting to know if their limiting behavior in regards to categoricity relative to a degree already occurs at the level of $\mathbf{0}$.

\printbibliography

\end{document}